\documentclass{amsart}
\usepackage[utf8]{inputenc}
\usepackage{a4wide}
\usepackage{amsmath,amsfonts,amssymb}
\usepackage{amsthm}
\usepackage{hyperref}
\usepackage{mathtools}

\newcommand{\Z}{\mathbb{Z}}
\newcommand{\Q}{\mathbb{Q}}
\newcommand{\R}{\mathbb{R}}
\newcommand{\C}{\mathbb{C}}
\newcommand{\N}{\mathbb{N}}
\newcommand{\ind}{\mathbf{1}}

\renewcommand{\Re}{\mathrm{Re}}
\renewcommand{\Im}{\mathrm{Im}}
\newcommand*\diff{\mathop{}\!\mathrm{d}}

\newcommand{\fdg}{\colon}
\newcommand{\eps}{\varepsilon}
\newcommand{\bb}{,\ldots ,}

\newtheorem*{theorem*}{Theorem}
\newtheorem{thm}{Theorem}
\newtheorem{lem}{Lemma}
\newtheorem{cor}{Corollary}

\theoremstyle{definition}
	\newtheorem{problem}{Problem}
	\newtheorem{rem}{Remark}

	\newcounter{countD}
	
	\newtheorem{problemStar}[countD]{Problem}

\makeatletter
\@namedef{subjclassname@2020}{%
  \textup{2020} Mathematics Subject Classification}
\makeatother

\begin{document}
\title[Pillai's problem with transcendental numbers]{On a variant of Pillai's problem with transcendental numbers}
\subjclass[2020]{11D75, 11D61, 11D45} 
\keywords{Diophantine inequalities, Pillai’s problem}
\thanks{
The first author was supported by the Austrian Science Fund (FWF) under the project I4406 and the project (SFB) F5510. The second and the last author were supported by the Austrian Science Fund (FWF) under the project I4406. The third author was also supported by the Austrian Science Fund (FWF) under the project W1230.
}

\author[R. Tichy]{Robert Tichy}
\address{R. Tichy,
Institute of Analysis and Number Theory, Graz University of Technology,
Kopernikusgasse 24/II, 
A-8010 Graz, Austria
}
\email{tichy\char'100tugraz.at}

\author[I. Vukusic]{Ingrid Vukusic}
\address{I. Vukusic,
University of Salzburg,
Hellbrunnerstrasse 34/I,
A-5020 Salzburg, Austria}
\email{ingrid.vukusic\char'100plus.ac.at}

\author[D. Yang]{Daodao Yang}
\address{D. Yang,
Institute of Analysis and Number Theory, Graz University of Technology,
Kopernikusgasse 24/II, 
A-8010 Graz, Austria
}
\email{yang\char'100tugraz.at}

\author[V. Ziegler]{Volker Ziegler}
\address{V. Ziegler,
University of Salzburg,
Hellbrunnerstrasse 34/I,
A-5020 Salzburg, Austria}
\email{volker.ziegler\char'100plus.ac.at}

\begin{abstract}
In this paper, we study the 
asymptotic behaviour of the number of solutions $(m, n)\in \N^2$ to the  inequality $ | \alpha^n - \beta^m | \leq x $  when $x$ tends to  infinity.  Here $\alpha, \beta$ are  given multiplicatively independent complex  numbers  with $|\alpha| > 1$ and $|\beta|>1$.
\end{abstract}

\maketitle

\section{Introduction}

Pillai considered gaps between perfect powers since the 1930s. He proved  \cite{PillaiIneq} that for fixed positive integers $a, b$ with $a > 1$ and $b >1$ the following holds:
\[ \# \{ (n,m)\in \N^2 \colon 0 < a^n-b^m \leq x \} \sim \frac{(\log x)^2}{2 \log a \cdot \log b },	\quad \text{as } x \to \infty.
\]

Furthermore, Pillai \cite{Pillai, Pillaicorrection} proved that the number of positive integers which are smaller than $x$ and can be expressed in the form $a^n - b^m$ is asymptotically equal to $ (\log x)^2/ (2 \log a \cdot \log b)$, when $x$ goes to  infinity.

Pilla's results on perfect powers have been generalized to linear recurrence sequences with certain dominant root conditions in \cite{Yang2020}, and for more general cases in \cite{TichyVukusicYangZiegler2021}. With exactly the same methods it can be proved that for  multiplicatively independent complex algebraic numbers $\alpha, \beta$ with $|\alpha| > 1$ and $|\beta|>1$ we have the following:
\begin{equation}\label{eq:main}
	\#\{(n, m)\in \N^2\colon ~  | \alpha^n - \beta^m | \leq x \}  \sim \frac{(\log x)^2}{\log |\alpha|\cdot \log  |\beta| },
	\quad \text{as } x \to \infty.
\end{equation}

It was conjectured in \cite{Yang2020} that \eqref{eq:main} might hold even if we do not assume $\alpha $ and $\beta $ to be algebraic. However, formula \eqref{eq:main} does not hold for general transcendental numbers $\alpha$ and $\beta$. In Section \ref{sec:counterex} we construct a counterexample.

Yet, the following problem remains unsolved (cf.\ \cite{TichyVukusicYangZiegler2021}).

\begin{problem}\label{probl:main}
For which multiplicatively independent complex numbers $\alpha,\beta \in \C$ with $|\alpha|>1$ and $|\beta| > 1$ does formula \eqref{eq:main} hold?
\end{problem}

It turns out that in fact formula \eqref{eq:main} holds for almost all $\alpha$ and $\beta$. This is proved in Section~\ref{sec:metric}, where we give a sufficient condition for \eqref{eq:main} to hold. The condition involves the irrationality exponent of $\log |\alpha| / \log |\beta|$. However, it is in general difficult to decide, whether the condition is fulfilled for specific $\alpha$ and $\beta$. Moreover, the result contains some ineffective constants.
In Section \ref{sec:specialCases} we prove effective results for two special cases, where $\alpha$ and/or $\beta$ are transcendental of the form $e^\gamma$ with algebraic $\gamma$. In Section \ref{sec:aProblem} we consider a variant of a problem earlier posed by the authors \cite[Problem 2]{TichyVukusicYangZiegler2021}.

\section{A counterexample}\label{sec:counterex}

In this section we find two real numbers $\alpha$, $\beta$ with $|\alpha|>1$ and $|\beta|>1$, which are multiplicatively independent but for which the inequality
\begin{equation}\label{eq:counterI}
|\alpha^n - \beta^m| \leq 1
\end{equation}	
has infinitely many solutions $(n,m)\in \N^2$. This shows that formula \eqref{eq:main} does not hold for all real numbers $\alpha$ and $\beta$.

The main idea is to construct an extremely well approximable Liouville number. We define
\[
	c= \sum_{i=0}^\infty 10^{-a(i)},
\]
where $a(i)$ is given by
\begin{align*}
	&a(0)=1 \quad \text{and} \\
	&a(i+1)=10^{a(i)}
	\quad \text{for } i \geq 0.
\end{align*}

Then by construction there are infinitely many approximations $\frac{p_k}{q_k}=\sum_{i=0}^k 10^{-a(i)}$ with
\[
	\left| c - \frac{p_k}{q_k} \right|
	= \sum_{i=k+1}^\infty 10^{-a(i)}
	< 2 \cdot 10^{-a(k+1)}
	= 2 \cdot 10^{-10^{a(k)}}.
\]
Note that $q_k=10^{a(k)}$ and therefore we have
\begin{equation}\label{eq:ineq_qn}
	\left| c - \frac{p_k}{q_k} \right| 
	< 2 \cdot 10^{-q_k}.
\end{equation}
Now we put
\[
	\beta=2
	\quad \text{and} \quad 
	\alpha=2^c.
\]
Clearly, $\alpha$ and $\beta$ are larger than 1 and multiplicatively independent. Moreover,
\[
	\frac{\log \alpha}{\log \beta} =c.
\]
Inserting into \eqref{eq:ineq_qn} we get that there are infinitely many $(n,m)=(p_k, q_k) \in \N^2$ with
\[
	\left| \frac{\log \alpha}{\log \beta} - \frac{m}{n} \right|
	< 2 \cdot 10^{-n}.
\]
This implies
\[
	|n \log \alpha - m \log \beta |
	< 2 \cdot \log \beta \cdot n \cdot 10^{-n}.
\]
Since $2 \cdot \log \beta \cdot n \cdot 10^{-n}<0.5$ and $|e^x-1|< 2|x|$ for $0<|x|<0.5$ we obtain
\[
	\left|\frac{\alpha^n}{\beta^m} - 1 \right|
	< 4 \cdot \log \beta \cdot n \cdot 10^{-n}
\]
for infinitely many $(n,m)$. This is equivalent to
\begin{equation}\label{eq:ready}
	|\alpha^n - \beta^m | < \frac{4 \cdot \log \beta \cdot n \cdot \beta^m }{10 ^n}.
\end{equation}
Note that $\frac{m}{n}\approx c \approx 0.1$, so $m<n$. Moreover, recall that $\beta =2$. Therefore, the right hand side of \eqref{eq:ready} becomes arbitrarily small and in particular inequality \eqref{eq:counterI} is satisfied
for infinitely many $(n,m)\in \N^2$.

\section{Auxiliary results}

In the subsequent sections we will heavily rely on results from Diophantine approximation, including results on the irrationality exponent and lower bounds for linear forms in logarithms. We will also need a simple lemma from measure theory and some elementary inequalities. These results are stated in this section. 

We also define $T_{\alpha,\beta}(x)$ as the number of solutions $(n,m)$ in formula \eqref{eq:main} in order to simplify notation.

\subsection{The irrationality exponent}

We recall the classical definition of the irrationality exponent $\mu (\xi)$ of a real number $\xi$ (see for instance \cite[Appendix E]{Bugeaud2012}):
\[
	\mu(\xi)
	= \sup \left\{\mu:  0 < |\xi -\frac{p}{q}|< \frac{1}{q^\mu} 
	~\text{has infinitely many solutions}~ (p,q)\in \Z \times \N \right\}.
\]
The irrationality exponent for a rational number is 1 and for a Liouville number is $\infty$ (by definition of a Liouville number). For all other real numbers $\xi$, we have $2 \leq \mu(\xi) < \infty $. Since the rational numbers and the Liouville numbers have measure 0 in $\R$ (with respect to the Lebesgue measure), almost all real numbers $\xi$ have an irrationality exponent $2 \leq \mu(\xi) < \infty $.

Liouville's approximation theorem says that the irrationality exponent of any real algebraic number is at most its degree. In fact, by Roth's theorem the irrationality exponent of any real algebraic number is equal to 2.
For transcendental numbers it is in general very hard to determine the irrationality exponent. However, a metric result exists: With respect to the Lebesgue measure, almost all real numbers have an irrationality exponent equal to 2 (see for instance \cite[Theorem E.3]{Bugeaud2012}).

Let us finally note that by the definition of the irrationality exponent, for any given irrational real number $\xi$ and $\eps > 0$, there exists a positive constant $c(\xi, \eps)$ such that
\[
	\left|\xi - \frac{p}{q}\right| \geq \frac{c(\xi,\eps)}{q^{\mu(\xi) + \eps}}
	\quad \text{for all }  (p,q) \in \Z \times \N.
\]
However, results such as Roth's theorem are ineffective, i.e.\ the constant $c(\xi, \eps)$ cannot in general be determined effectively.

\subsection{Lower bounds for linear forms in logarithms}

In order to obtain an effective result for special cases of Problem~\ref{probl:main}, we will use Waldschmidt's bound for inhomogeneous linear forms in logarithms \cite{Waldschmidt2000}.

Let us first recall the definition of the \textit{logarithmic height}. 
Let $\gamma$ be an algebraic number of degree $d\geq 1$ with the minimal polynomial
\[
	a_d X^d + \dots +a_1 X + a_0 
	= a_d \prod _{i=1}^{d} (X- \gamma_i),
\]
where $a_0, \dots, a_d$ are relatively prime integers and $\gamma_1, \dots, \gamma_d$ are the conjugates of $\gamma$. Then the logarithmic height of $\gamma$ is given by
\[
	h(\gamma)
	= \frac{1}{d} \left(
		\log |a_d|
		+ \sum_{i=1}^d \log \left( \max \{ 1,|\gamma_i|\} \right)
		\right).
\]

Now we state Waldschmidt's theorem. We have adapted it from \cite[Thm. 2.1]{Bugeaud2018}.

\begin{lem}[Waldschmidt]\label{lem:waldschmidt}
Let $t\geq 1$ and $\alpha_1 \bb \alpha_t$ be non-zero algebraic numbers. Let $\log \alpha_1 \bb \log \alpha_t$ be determinations of their logarithms and assume that $\log \alpha_1 \bb \log \alpha_t$ are linearly independent over $\Q$. Let $\beta_0\bb\beta_t$ be algebraic numbers, not all zero. 
Then we have
\begin{multline*}
	\log |\beta_0 + \beta_1 \log \alpha_1 + \dots + \beta_t \log \alpha_t |
	\geq -2^{t+25} t^{3t+9} D^{t+2} \log A_1 \cdots \log A_t 
		\log B \log E,
\end{multline*}
where $D$ is the degree of the number field $\Q(\alpha_1\bb \alpha_t, \beta_0 \bb \beta_t)$ over $\Q$ and $E, A_1 \bb A_t$ and $B$ are real numbers with
\begin{align*}
	E 
	& \geq \max\{e^{1/D},D\},\\
	\log A_j
	&\geq \max \left\{h(\alpha_j), \frac{e}{D}|\log \alpha_j|, \frac{1}{D}\right\}
	\quad \text{for } 1 \leq j \leq t,\\
	B &\geq \max\left\{
		E,
		\max_{1\leq j \leq t} D \log A_j,
	{\max_{0 \leq j \leq t} e^{h(\beta_j)}}
	\right\}. 
\end{align*}
\end{lem}

\subsection{A lemma from measure theory and some auxiliary inequalities}

\begin{lem}\label{lem:measure}
Let $\Gamma$ be a null set in $\R$ (with respect to the Lebesgue measure). Then the set of all $(\alpha,\beta) \in \C^2$ with $|\alpha|,|\beta|>1$ and $\log |\alpha| / \log |\beta| \in \Gamma$ is a null set in $\C^2$ (with respect to the 4-dimensional Lebesgue measure).
\end{lem}
\begin{proof}
First, let 
\[
	A=\left\{ (x,y)\in \R_{>1}^2 \fdg \log x / \log y \in \Gamma \right\}.
\]
We want to show that $A$ is a null set in $\R^2$, i.e.\
\[
	\lambda(A)	
	= \int_{\R_{>1}^2}\ind_A(x,y) \diff^2(x,y)
	=0,
\]
where $\ind_A(x,y)$ is the indicator function corresponding to $A$. 
By Fubini's Theorem we have
\[
	\int_{\R_{>1}^2} \ind_A(x,y) \diff^2(x,y)
	= \int_{\R_{>1}} \int_{\R_{>1}} \ind_A(x,y) \diff x \diff y.
\]
For fixed $y>1$ the function $x \mapsto \log x /\log y$ is differentiable and therefore the set of all $x$ for which $\log x /\log y \in \Gamma$ is a null set. Thus, in the above integral, the inner integral is always zero and so the whole integral is zero and $A$ is indeed a null set.

Now in order to prove the lemma, we need to show that
\[
	B = \{(\alpha,\beta)\in \C^2 \fdg (|\alpha|,|\beta|)\in A\}
\]
is a null set in $\C^2$ (seen as $\R^4$ with the Lebesgue measure). Let $\ind_B(\alpha_1, \alpha_2, \beta_1, \beta_2)$ be the indicator function of $B$, i.e.\
\[
	\ind_B(\alpha_1, \alpha_2, \beta_1, \beta_2)
	= \begin{cases}
	1 & \text{if } (|\alpha_1 + i \alpha_2|,|\beta_1 + i \beta_2|) \in A,\\
	0 & \text{otherwise}.
	\end{cases}
\]
Then for the Lebesgue measure of $B$ we have
\begin{align*}
	\lambda(B)
	&= \int_{\R^4} \ind_B(\alpha_1, \alpha_2, \beta_1, \beta_2) \diff^4 (\alpha_1, \alpha_2, \beta_1, \beta_2)\\
	&= \int_0^\infty \int_0^{2\pi} \int_0^\infty \int_0^{2\pi}
		\ind_A (r_\alpha, r_\beta) r_\alpha r_\beta 
		\diff \varphi_\alpha \diff r_\alpha \diff \varphi_\beta \diff r_\beta \\
	&= \int_0^{2\pi} \int_0^{2\pi} \left(
			\int_0^\infty \int_0^\infty 
		\ind_A (r_\alpha, r_\beta) r_\alpha r_\beta 
		\diff r_\alpha \diff r_\beta \right) 
			\diff \varphi_\alpha \diff \varphi_\beta,
\end{align*}
where we used Fubini's theorem, polar coordinates and the definition of $B$. Now since $A$ is a null set in $\R^2$, clearly the inner part of the last integral is zero, so indeed $\lambda(B)=0$.
\end{proof}

\begin{lem}\label{lem:ineq}
Let $k$, $c$ and $d$ be positive constants. 
Suppose that $n \geq N=N(k,c,d)$ is a large number (to be precise, we need $n$ to satisfy $c \log n \geq 2$ and $n \geq k^2 c^2 (d+2)^2e^{d/c}(\log n)^2$). Suppose that
\begin{equation}\label{eq:ineq_assump}
	n \leq kz + c \log n + d
\end{equation}
for some $z\geq 2/k$.
Then
\[
	n \leq kz + 2c \log z.
\]
\end{lem}
\begin{proof}
Note that for $r,s\geq 2$ we have $\log(r+s)\leq \log r + \log s$. By assumption, $n$ and $z$ are large, in particular $kz \geq 2$ and $c\log n \geq 2$. Thus, using assumption \eqref{eq:ineq_assump}, we have
\begin{align*}
	\log n
	&\leq \log (kz + c \log n + d)\\
	&\leq \log (kz) + \log (c\log n) + \log (d+2)\\
	&= \log k + \log z + \log c + \log \log n + \log (d+2)\\
	&\leq \log z + \frac{1}{2}\log n - \frac{d}{2c}
\end{align*}
for $n\geq N$. Thus $\log n \leq 2 \log z - \frac{d}{c}$ and using assumption \eqref{eq:ineq_assump} again we get
\begin{align*}
	n 
	&\leq kz + c \log n + d\\
	&\leq kz + c \left( 2 \log z  - \frac{d}{c} \right) + d\\
	&=kz + 2 c \log z.
\end{align*}
\end{proof}

Next, we make a definition in order to simplify notation when dealing with Problem \ref{probl:main}.

Let $\alpha,\beta \in \C$ with $|\alpha|>1$ and $|\beta|>1$. We define
\[
	T_{\alpha,\beta}(x)
	= \#\{(n, m) \in \N^2 \fdg | \alpha^n - \beta^m | \leq x\}.
\]

For arbitrary $\alpha$ and $\beta$ there is a trivial lower bound for $T_{\alpha,\beta}(x)$:

\begin{lem}\label{lem:lowerBound}
Let $\alpha,\beta \in \C$ with $|\alpha|>1$ and $|\beta|>1$. Then
\[
	T_{\alpha,\beta}(x)
	\geq \frac{(\log x)^2}{\log |\alpha|\cdot \log  |\beta| } + O(\log x)
\]
and the implied constants are effectively computable.
\end{lem}
\begin{proof}
Suppose that
\begin{align*}
	n\leq \frac{\log x-\log 2}{\log |\alpha|}
	\quad \text{and} \quad
	m\leq \frac{\log x-\log 2}{\log |\beta|}.
\end{align*}
Then 
\begin{align*}
	|\alpha^n - \beta^m|
	\leq |\alpha|^n + |\beta|^m 
	\leq \frac{x}{2}+\frac{x}{2}
	=x.
\end{align*}
But the number of such pairs $(n,m)$ is greater than
\[
	\left(\frac{\log x-\log 2}{\log |\alpha|}-1\right)
	\left(\frac{\log x-\log 2}{\log |\beta|}-1\right)
	= \frac{(\log x)^2}{\log |\alpha|\cdot \log  |\beta| } + O(\log x),
\]
which proves the lower bound for $T_{\alpha,\beta}(x)$.
\end{proof}

\begin{lem}\label{lem:absVal}
For any $\alpha,\beta \in \C$ with $|\alpha|>1$ and $|\beta|>1$ we have
\[
	T_{\alpha,\beta}(x) \leq T_{|\alpha|,|\beta|}(x).
\]
\end{lem}
\begin{proof}
This follows immediately from the reverse triangle inequality:
\[
	\left| |\alpha|^n - |\beta|^m \right|
	\leq \left| \alpha^n - \beta^m \right|.
\]
\end{proof}

\section{A sufficient condition and a metric result}\label{sec:metric}
In this section, we first prove a sufficient condition for formula \eqref{eq:main} to be true. Then we show that formula \eqref{eq:main} is true for almost all $(\alpha,\beta)\in \C^2$ with $|\alpha|>1$ and $|\beta|>1$. 

\begin{thm}\label{thm:sufficient}
Let $(\alpha,\beta) \in \C^2$ with $|\alpha|>1$ and $|\beta|>1$. Assume that   $2 \leq \mu(\frac{\log |\alpha|}{\log |\beta|}) < \infty $, then we have
\[
	T_{\alpha,\beta}(x)
	= \frac{(\log x)^2}{\log |\alpha| \log |\beta|} + O(\log x \cdot \log \log x)
\]
for $x$ large enough. 

\end{thm}

\begin{proof}
By Lemma \ref{lem:lowerBound} we have
\[
	T_{\alpha,\beta}(x) 
	\geq \frac{(\log x)^2}{\log |\alpha|\log |\beta|} + O(\log x)
\]
for all $\alpha,\beta$.
Therefore, we only need to show that
\[
	T_{\alpha,\beta}(x)
	\leq \frac{(\log x)^2}{\log |\alpha| \log |\beta|} + O(\log x \cdot \log \log x).
\]

By Lemma \ref{lem:absVal} we may assume that $\alpha, \beta$ are real and positive.
Suppose that 
\begin{equation}\label{eq:metric_main}
	|\alpha^n-\beta^m|\leq x.
\end{equation}
We will show that  $n \leq  \frac{\log x}{\log \alpha}  + O(\log \log x)$ and $m\leq \frac{\log x}{\log \beta} + O(\log \log x)$. Then counting the possible solutions $(n,m)$ will yield the required upper bound for $T_{\alpha,\beta}(x)$.

If $\alpha^n\leq 2x$, then we immediately obtain 
\[
	n 
	\leq \frac{\log x}{\log \alpha} + \frac{\log 2}{\log \alpha}
	= \frac{\log x}{\log \alpha} + O(\log \log x).
\]	
From now on assume that
$\alpha^n > 2x$.

We divide \eqref{eq:metric_main} by $\alpha^n$, obtaining
\[
	\left| \frac{\beta^m}{\alpha^n} - 1 \right|
	\leq \frac{x}{\alpha^n}.
\]
By our assumption, the expressions above are $<0.5$.
Since $|\log (y+1)|\leq 2|y|$ for $|y|<0.5$ we have
\begin{equation*}
	\left| m \log \beta - n \log \alpha\right|
	\leq \frac{2x}{\alpha^n}.
\end{equation*}
Dividing by $n$ and $\log \beta$ we obtain
\begin{equation}\label{eq:metric_beforeKhinch}
	\left| \frac{\log \alpha}{\log \beta} - \frac{m}{n} \right|
	\leq \frac{2x}{\log \beta \cdot n \cdot \alpha^n}.
\end{equation}
Denote $\mu = \mu(\log \alpha/\log \beta)$, then we have
\[
	\left|\frac{\log \alpha}{\log \beta} - \frac{m}{n}\right| \geq \frac{c(\alpha,\beta)}{n^{\mu + 0.1}}
	\quad \text{for all }  (m,n) \in \Z \times \N.
\]
Combining this with \eqref{eq:metric_beforeKhinch} we obtain
\[
	\frac{c(\alpha,\beta)}{n^{\mu+0.1}}
	\leq \frac{2x}{\log \beta \cdot n \cdot \alpha^n}.
\]
This implies
\[
	\alpha^n
	\leq \tilde{c}(\alpha,\beta)\cdot x\cdot n^{\mu -0.9}
\]
and thus
\[
	n
	\leq \frac{\log x}{\log \alpha} + \frac{\mu -0.9}{\log \alpha}\log n + \hat{c}(\alpha,\beta).
\]
By Lemma \ref{lem:ineq} with $z=\log x$, $k=1/\log \alpha$, $c=(\mu -0.9)/\log \alpha$ and $d= \hat{c}(\alpha,\beta)$ this implies 
\begin{equation}\label{eq:metric_n}
	n
	\leq \frac{\log x}{\log \alpha} + O(\log \log x),
\end{equation}
where the implied constant depends on $\hat{c}(\alpha,\beta)$.

By completely analogous arguments (note that $2 \leq \mu(\frac{\log \alpha}{\log \beta}) < \infty $ if and only if $2 \leq \mu(\frac{\log \beta}{\log \alpha}) < \infty $) we also get
\begin{equation}\label{eq:metric_m}
	m\leq \frac{\log x}{\log \beta} + O(\log \log x).
\end{equation}

Therefore, all solutions $(n, m)$ to the Diophantine inequality $|\alpha^n - \beta^m| \leq x$ 
have the properties \eqref{eq:metric_n} and \eqref{eq:metric_m}.
But there are at most
\[
	\frac{(\log x)^2}{\log \alpha \log \beta} + O(\log x \cdot \log \log x)
\]
such solutions. Thus
\[
	T_{\alpha,\beta}(x)
	\leq \frac{(\log x)^2}{\log \alpha \log \beta} + O(\log x \cdot \log \log x),
\]
which completes the proof of Theorem \ref{thm:sufficient}.
\end{proof}

\begin{cor}
For almost all $(\alpha,\beta) \in \C^2$ (with respect to the 4-dimensional Lebesgue measure) with $|\alpha|>1$ and $|\beta|>1$ we have
\begin{equation}\label{eq:cor}
	T_{\alpha,\beta}(x)
	= \frac{(\log x)^2}{\log |\alpha| \log |\beta|} + O(\log x \cdot \log \log x)
\end{equation}
for $x$ large enough.
\end{cor}
\begin{proof}
Let $\Gamma = \Q \cup \{\xi \in \R \fdg \mu(\xi) = \infty \}$. This is a null set in $\R$ with respect to the Lebesgue measure.
By Lemma~\ref{lem:measure}, the set  $\{ (\alpha,\beta)\in \C_{|\cdot|>1}^2 \fdg \log |\alpha| / \log |\beta| \in \Gamma \}$ is a null set in $\C^2$. Thus, the condition $2 \leq \mu({\log |\alpha|}/{\log |\beta|}) < \infty $ from Theorem \ref{thm:sufficient} is satisfied by almost all $(\alpha,\beta) \in \C^2$, so formula \eqref{eq:cor} indeed holds for almost all $(\alpha,\beta) \in \C^2_{|\cdot|>1}$.
\end{proof}

\section{Some special cases}\label{sec:specialCases}

First, we consider the case where $\alpha$ is algebraic and $\beta$ is transcendental of the form
$\beta=e^\gamma$, where $\gamma$ is an algebraic number.

\begin{thm}\label{thm:aeb}
Let $\alpha$ and $\gamma$ be two algebraic numbers and suppose that  $|\alpha| >1$ and $\Re(\gamma)>0$.
Let
\[
	T(x)=\# \{ (n,m)\in \N^2 \fdg |\alpha^n-(e^\gamma)^m|\leq x \}.
\]
Then we have
\[
	T(x)
	= \frac{(\log x)^2}{(\log |\alpha|) \cdot \Re(\gamma)} + O(\log x \cdot \log \log x)
\]
for $x$ large enough
and the implied constants are effectively computable.
\end{thm}

\begin{proof}
Since $\log |e^\gamma|=\Re(\gamma)$, by Lemma \ref{lem:lowerBound} we have
\[
	T(x) \geq \frac{(\log x)^2}{(\log |\alpha|) \cdot \Re(\gamma)} + O(\log x).
\]
Therefore, we only need to show that
\[
	T(x)
	\leq \frac{(\log x)^2}{(\log |\alpha|) \cdot \Re(\gamma)} + O(\log x \cdot \log \log x)
\]
and by Lemma \ref{lem:absVal} we may assume that $\alpha$ and $e^\gamma$ are real and positive, which means that $\gamma$ is real and positive as well.

Suppose that 
\begin{equation}\label{eq:aeb_main}
	|\alpha^n-(e^\gamma)^m|\leq x.
\end{equation}

We will show that $n \leq \frac{\log x}{\log \alpha} + O(\log \log x)$ and $m\leq \frac{\log x}{\gamma}+ O(\log \log x)$. Then counting the possible solutions $(n,m)$ yields the required upper bound for $T(x)$.

We distinguish between two cases.

\noindent\textbf{Case 1:}
$\alpha^n\leq 2x$.  This immediately yields 
\[
	n 
	\leq \frac{\log x}{\log \alpha} + \frac{\log 2}{\log \alpha}
	= \frac{\log x}{\log \alpha} + O(\log \log x).
\]	
Moreover, inequality \eqref{eq:aeb_main} implies
\[
	e^{\gamma m} 
	\leq x + \alpha^n 
	\leq 3x,
\]
which yields 
\[
	m  
	\leq \frac{\log 3x}{\gamma}
	= \frac{\log x}{\gamma} + \frac{\log 3}{\gamma}
	= \frac{\log x}{\gamma} + O(\log \log x),
\]
as required.

\noindent\textbf{Case 2:} 
$\alpha^n > 2x$. 
Inequality \eqref{eq:aeb_main} implies
\[
	e^{\gamma m} 
	\leq x+ \alpha^n 
	< 3 \alpha ^n,
\]
which implies
\begin{equation}\label{eq:aeb:mlessn}
	m < n \frac{\log \alpha}{\gamma}+ \frac{\log 3}{\gamma}.
\end{equation}
In particular, we estimate $m < C_{1} n$ with some positive constant $C_1$.

We divide \eqref{eq:aeb_main} by $\alpha^n$, obtaining
\begin{equation}\label{eq:aeb_I}
	\left| \frac{e^{\gamma m}}{\alpha^n}-1 \right|
	\leq \frac{x}{\alpha^n}.
\end{equation}
By the case assumption the expressions above are $<0.5$.
Since $|\log (y+1)|\leq 2|y|$ for $|y|<0.5$ we have
\begin{equation}\label{eq:aeb_II}
	|\Lambda| \coloneqq
	| \gamma m - n \log \alpha|
	\leq \frac{2x}{\alpha^n}.
\end{equation}

Next we apply Waldschmidt's theorem with the parameters $t=1$, $\beta_0= \gamma m$, $\beta_1=-n$ and $\alpha_1=\alpha$. Let $D$ be the degree of $\Q(\alpha, \gamma)$ over $\Q$ and set
\begin{align*}
	E 
	& = \max\{e^{1/D},D\},\\
	\log A_1 
	&= \max \left\{h(\alpha), \frac{e}{D}\log \alpha, \frac{1}{D}\right\},\\
	B &= C_{2} n \geq e^{h(\gamma)} C_{1}n.
\end{align*}
Note that $h(\beta_0)=h(\gamma m) \leq h(\gamma) + h(m)=h(\gamma) +  \log m\leq h(\gamma) + \log (C_{1}n)$ and $h(\beta_1)=h(n)=\log n$, 
so if $C_2$ is chosen large enough, we have indeed
\[
	B \geq \max \{E, D \log A_1, e^{h(\beta_0)}, e^{h(\beta_1)}\}.
\]
Now Waldschmidt's theorem tells us that
\[
	\log |\Lambda|
	\geq -C_3 \log B 
	= - C_3 (\log C_2 + \log n)
	\geq - C_4 \log n,
\]
where $C_3$ is the positive constant coming from Waldschmidt's theorem, only depending on $\alpha$ and $\gamma$.
 Combining this with \eqref{eq:aeb_II} we obtain
\[
	- C_4 \log n
	\leq \log |\Lambda|
	\leq \log 2 + \log x - n \log \alpha,
\]
which implies
\[
	n \leq \frac{\log x}{\log \alpha} + \frac{C_4}{\log \alpha} \log n + \frac{\log 2}{\log \alpha}.
\]
By Lemma \ref{lem:ineq} with $z=\log x$, $k=1/\log \alpha$, $c=C_4/\log \alpha$ and $d= \log 2 / \log \alpha$ this implies
\[
	n 
	\leq \frac{\log x}{\log \alpha} + O( \log \log x).
\]
Inserting into inequality \eqref{eq:aeb:mlessn} we also get
\[
	m \leq \frac{\log x}{\gamma} + O(\log \log x),
\]
which concludes the proof.
\end{proof}

The second special case we want to treat is where $\alpha$ and $\beta$ are both transcendental and of the form $e^{\gamma}$ and $e^{\delta}$, where $\gamma$ and $\delta$ are algebraic numbers.

\begin{thm}\label{thm:twoExponential}
Let $\gamma$ and $\delta$ be two algebraic numbers with $\Re(\gamma)>0$ and $\Re(\delta)>0$, which are linearly independent over $\Q$.
Let
\[
	T(x)
	= \# \{(n,m)\in \N^2 \fdg |(e^\gamma)^n - (e^\delta)^m| \leq x \}.
\]
Then we have
\[
	T(x)=\frac{(\log x)^2}{\Re(\gamma) \cdot \Re(\delta)} + O(\log x \cdot \log \log x)
\]
for $x$ large enough
and the implied constants are effective.
\end{thm}
\begin{proof}
Since $\log |e^\gamma|=\Re(\gamma)$ and $\log |e^\delta|=\Re(\delta)$, by Lemma \ref{lem:lowerBound} we have
\[
	T(x) \geq \frac{(\log x)^2}{\Re(\gamma) \cdot \Re(\delta)} + O(\log x).
\]
Therefore, we only need to show that
\[
	T(x)
	\leq \frac{(\log x)^2}{\Re(\gamma) \cdot \Re(\delta)} + O(\log x \cdot \log \log x).
\]
Suppose that 
\begin{equation}\label{eq:ineq_ean-ebm}
	|(e^\gamma)^n - (e^\delta)^m| \leq x.
\end{equation}
We will show that $n\leq \frac{\log x }{\Re(\gamma)} + O(\log \log x)$ and $m\leq \frac{\log x}{\Re(\delta)} + O(\log \log x)$. Then counting the possible solutions $(n,m)$ yields the required upper bound for $T(x)$.

We only prove the bound for $n$. The proof of the bound for $m$ is completely analogous. 
We distinguish between two cases.

\noindent\textbf{Case 1:} $|e^{\gamma n}|\leq 2x$. This immediately yields
\[
	n \leq \frac{\log x }{\Re(\gamma)} + \frac{\log 2}{\Re(\gamma)}
	= \frac{\log x}{\Re(\gamma)} + O(\log \log x).
\]	

\noindent\textbf{Case 2:}
$|e^{\gamma n}| > 2x$. 
Inequality \eqref{eq:ineq_ean-ebm} implies
\[
	|e^\delta|^m
	\leq x +  |e^\gamma|^n
	< 3 |e^\gamma|^n,
\]
so $m \leq C_5 n$ for some positive constant $C_5$ only depending on $\gamma$ and $\delta$.

We divide \eqref{eq:ineq_ean-ebm} by $|e^{\gamma}|^n$, obtaining
\begin{equation}\label{eq:ineq_eqn-ebm_divided}
	\left| \frac{e^{\delta m}}{e^{\gamma n}} - 1 \right| 
	\leq \frac{x}{|e^{\gamma}|^n}.
\end{equation}
Let $\log(z)$ be the principal branch of the complex logarithm, with $\Im (\log (z)) \in [-\pi,\pi)$. Then, since $\log(1+z)=z-\frac{z^2}{2}+\frac{z^3}{3}-+\dots$ for $|z|<1$, we have $|\log (1+z)|\leq 2 |z|$ for $|z|<0.5$.
By the case assumption the expressions in \eqref{eq:ineq_eqn-ebm_divided} are $< 0.5$, so we get
\[
	\left|\log\left(e^{\delta m - \gamma n}\right)\right|	
	=| \delta m - \gamma n + 2k\pi i|
	\leq \frac{2x}{|e^{\gamma}|^n},
\]
where $k$ is an integer such that $| \Im(\delta) m-\Im(\gamma) n + 2k\pi| \leq \pi$. Note that this implies $|k| \leq C_6 \max\{n,m\} \leq C_7 n$. Set $\eta = \delta m - \gamma n$ and note that $\pi i = \log(-1)$. Then we can write
\begin{equation}\label{eq:egamma-edelta_lambda}
	|\Lambda|\coloneqq
	|\eta + 2k\log(-1)| 
	\leq \frac{2x}{|e^{\gamma}|^n}.
\end{equation}
Now we apply Waldschmidt's theorem with $t=1$, $\alpha_1=-1$, $\beta_0= \eta$ and $\beta_1=2k$. Let $D$ be the degree of $\Q(\gamma,\delta)$ over $\Q$ and set
\begin{align*}
	E 
	& = \max\{e^{1/D},D\},\\
	\log A_1 
	&= \frac{e \pi}{D}
	= \max \left\{h(\alpha_1), \frac{e}{D} |\log \alpha_1|, \frac{1}{D}\right\},\\
	B &= n^{C_{10}},
\end{align*}
where $C_{10}$ is a sufficiently large constant.
Note that $h(\beta_0)=h(\delta m - \gamma n) \leq h(\delta) + \log m +  h(\gamma) + \log n + \log 2 \leq C_8 \log n$ and $h(\beta_1)=h(2k)\leq \log 2 + \log |k| \leq \log 2 +\log(C_7n)\leq C_9 \log n$, so if $C_{10}$ is chosen large enough we have indeed
\[
	B \geq \max \{E, D \log A_1, e^{h(\beta_0)}, e^{h(\beta_1)}\}.
\]
Now Waldschmidt's theorem tells us that
\[
	\log |\Lambda|
	\geq -C_{11} \log B 
	= - C_{11} (C_{10}\log n)
	= - C_{12} \log n,
\]
where $C_{11}$ is the positive constant coming from Waldschmidt's theorem, only depending on $\gamma$ and $\delta$.
Combining this with \eqref{eq:egamma-edelta_lambda} we obtain
\[
	- C_{12} \log n
	\leq \log |\Lambda|
	\leq \log 2 + \log x - n\, \Re(\gamma),
\]
which implies
\[
	n \leq \frac{\log x}{\Re(\gamma)} + \frac{C_{12}}{\Re(\gamma)} \log n + \frac{\log 2}{\Re(\gamma)}.
\]
By Lemma \ref{lem:ineq} with $z=\log x$, $k=1/\Re(\gamma)$, $c=C_{12}/\Re(\gamma)$ and $d= \log 2 / \Re(\gamma)$ this implies
\[
	n 
	\leq \frac{\log x}{\Re(\gamma)} + O( \log \log x),
\]
which completes the proof.
\end{proof}

\begin{rem}
In the statement of Theorem \ref{thm:twoExponential}, if we make the stronger assumption that $\Re(\gamma)$ and $\Re(\delta)$ should be linearly independent over $\Q$ instead of supposing linear independence of $\gamma$ and $\delta$, then $\mu(\log |e^\gamma|/\log |e^\delta|) = \mu(\Re(\gamma)/\Re(\delta)) = 2$ (by Roth's theorem).  Thus in this case, the result  in Theorem~\ref{thm:twoExponential} immediately follows from  Theorem~\ref{thm:sufficient}. However, the constants in Theorem~\ref{thm:twoExponential} are effective, while in Theorem~\ref{thm:sufficient} they are not.
\end{rem}

\section{A problem with more than two powers}\label{sec:aProblem}
Similarly to the problem with integer recurrence sequences in \cite{Yang2020, TichyVukusicYangZiegler2021}, one can consider Diophantine inequalities with sequences of transcendental numbers satisfying dominant root conditions. For instance, consider the following problem, which is Problem A in  \cite{Yang2020} (resp.\ Problem 2 in \cite{TichyVukusicYangZiegler2021}).

\begin{problem}\label{probl:pi}
Is the following true?
\[
	\#\{(n, m)\in \N^2\colon ~  | \pi^n + (\sqrt 5)^n - 7^m - e^m| \leq x \}  \sim \frac{(\log x)^2}{\log \pi \cdot \log 7 },
	\quad \text{as } x \to \infty.
\]
\end{problem}

We cannot solve Problem \ref{probl:pi} because the dominant terms are $\pi^n$ and $7^m$. We can, however,  solve the following modified problem ($e$ and $\pi$ change positions), in which case the dominant terms are $e^n$ and $7^m$.

\setcounter{countD}{\getrefnumber{probl:pi}}\addtocounter{countD}{-1}

\begin{problemStar}
Is the following true?
\[
	\#\{(n, m)\in \N^2\colon ~  | e^n + (\sqrt 5)^n - 7^m - \pi^m| \leq x \}  \sim \frac{(\log x)^2}{ \log 7 },
	\quad \text{as } x \to \infty.
\]
\end{problemStar}

We now combine the methods from \cite{TichyVukusicYangZiegler2021} and Section \ref{sec:specialCases} and in order to give a positive answer to the above question. 

\begin{thm}
Let
\[
T(x) = 	\#\{(n, m)\in \N^2\colon ~  | e^n + (\sqrt 5)^n - 7^m - \pi^m| \leq x \}.
\]
Then
\[
	\frac{(\log x)^2}{\log 7}  + O(\log x) 
	\leq T(x) 
	\leq \frac{(\log x)^2}{\log 7} +  O(\log x \cdot \log \log x).
\]
\end{thm}

\begin{proof}
A lower bound for $T(x)$ can be given by a similar argument as in the proof of Lemma~\ref{lem:lowerBound}.
We now prove the upper bound. 
Assume that
\begin{equation}\label{eq:morePowers}
| e^n + (\sqrt 5)^n - 7^m - \pi^m| \leq x.
\end{equation}
As usual, we will show that $ n\leq \log x + O(\log \log x)$ and $m\leq \frac{\log x}{\log 7} + O(\log \log x)$, which immediately yields the required upper bound for $T(x)$.

Let us assume that 
\[
	e^n\leq 7^m
\]
(the case $e^n > 7^m$ can be treated completely analogously). Then we have
\begin{equation}\label{eq:assumpt_en7m}
	n \leq m \log 7 \leq 2m.
\end{equation} 
By the reverse triangle inequality, \eqref{eq:morePowers} implies, after division by $7^m$,
\begin{equation}\label{eq:moreP_1}
	\left|\frac{e^n}{7^m} - 1 \right|
	\leq \frac{(\sqrt{5})^n}{7^m} + \frac{\pi^m}{7^m} + \frac{x}{7^m}.
\end{equation}
We distinguish between two cases.

\noindent\textbf{Case 1:} $x\geq \max\{(\sqrt{5})^n, \pi^m\}$. Then  
\[
	\left| \frac{e^n}{7^m}-1 \right| 
	\leq  3 \cdot \frac{x}{7^m}.
\] This is basically the same situation as in \eqref{eq:aeb_I} and we get the desired bounds.

\noindent\textbf{Case 2:} $x< \max\{(\sqrt{5})^n, \pi^m\}$.
Then we can estimate \eqref{eq:moreP_1} further by 
\begin{align}\label{eq:moreP_gamma}
	\left|\frac{e^n}{7^m} - 1 \right|
	&\leq 2 \cdot \left( \frac{(\sqrt{5})^n}{7^m} + \frac{\pi^m}{7^m} \right)
	\leq 2 \cdot \left( \frac{(\sqrt{5})^n}{e^n} + \left( \frac{\pi}{7}\right)^m \right)\\
	&\leq 2 \cdot \left( \left(\frac{\sqrt{5}}{e}\right)^n + \left(\frac{\pi}{7}\right)^{n/2} \right)
	\leq 4 \cdot 0.83^n.\nonumber
\end{align}
Since $n\leq 2m$, Waldschmidt's theorem implies 
\[
	\log |n - m \log 7 | \geq -2^{26} \cdot e \cdot \log 7 \cdot \log (2 m)
\]
and combined with \eqref{eq:moreP_gamma} this yields 
\[
	n \leq C_0 \log m.
\]
Now we go back to \eqref{eq:morePowers}. Using the reverse triangle inequality and estimating we get
\[
	x 
	\geq 7^m - 2 \cdot e^n
	\geq 7^m - 2 \cdot e^{C_0 \log m}. 
\]
After taking logarithms and applying Lemma \ref{lem:ineq} we obtain
\[
	m \leq \frac{\log x}{\log 7} + O(\log \log x).
\] 
From \eqref{eq:assumpt_en7m} we also get
\[
	n \leq \log x + O(\log \log x).
\]
\end{proof}

\end{document}